\documentclass{amsart}
\usepackage{amscd,amsmath,amssymb}
\usepackage{pstricks}
\usepackage{latexsym,amsbsy,mathrsfs}
\usepackage{xy}
\usepackage{xypic}
\usepackage{pgf,tikz}

\newtheorem{thm}{Theorem}[section]
\newtheorem{lem}[thm]{Lemma}
\newtheorem{cor}[thm]{Corollary}
\newtheorem{prop}[thm]{Proposition}

\theoremstyle{definition}
\newtheorem{example}[thm]{Example}

\newtheorem{defn}[thm]{Definition}

\newtheorem{rem}[thm]{Remark}

\numberwithin{equation}{thm}

\begin{document}
\title[N-ANGULATED QUOTIENT CATEGORIES INDUCED BY MUTATION PAIRS]
{ N-ANGULATED QUOTIENT CATEGORIES INDUCED BY MUTATION PAIRS}

\author{Zengqiang Lin}
\address{ School of Mathematical sciences, Huaqiao University,
Quanzhou\quad 362021,  China.} \email{lzq134@163.com}

\thanks{This work was supported in part by the National Natural Science Foundation of China (Grants No. 11101084)
and the Natural Science Foundation of Fujian Province (Grants No. 2013J05009)}



\begin{abstract} We define mutation pair in an $n$-angulated category and prove that given such a mutation pair, the corresponding quotient category carries a natural $n$-angulated structure. This result generalizes a theorem of Iyama-Yoshino in classical triangulated category. As an application, we obtain that the quotient category of a Frobenius $n$-angulated category is also an $n$-angulated category.
\end{abstract}

\maketitle

\section{Introduction}
Triangulated categories were introduced by Grothendieck-Verdier \cite{[GV]} and Puppe \cite{[P]} independently to axiomatize the properties of derived categories and stable homotopy categories respectively. Triangulated category is a very important structure in both geometry and algebra.

Geiss, Keller and Oppermann introduced the notion of $n$-angulated categories, which are $``$higher dimensional" analogues of triangulated categories, and showed that certain $(n-2)$-cluster tilting subcategories of a triangulated category give rise to $n$-angulated categories. For $n=3$, an $n$-angulated category is nothing but a  classical triangulated category. Nowadays the theory of $n$-angulated categories has been developed further.
Bergh and Thaule discussed the axioms for an $n$-angulated category \cite{[BT1]}. They  introduced a higher $``$octahedral axiom" and showed that it is equivalent to the mapping cone axiom. Another examples of $n$-angulated categories from local algebras were given in \cite{[BT2]}.  The notion of Grothendieck group of an $n$-angulated category was introduced to give a  generalization of Thomason's classification theorem for triangulated subcategories \cite{[BT3]}. Recently, Jasso introduced $n$-abelian and $n$-exact categories, and showed that the quotient category of a Frobenius $n$-exact category has a natural structure of $(n+2)$-angulated category \cite[Theorem 5.11]{[J]}, which is a higher analogue of  Happel's Theorem \cite[Theorem 2.6]{[Ha]}.

The aim of this paper is to discuss the construction of $n$-angulated categories. In this paper, we define mutation pair in an $n$-angulated category and prove that given such a mutation pair, the corresponding quotient category carries a natural $n$-angulated structure.  For $n=3$, our main result recover Iyama-Yoshino's theorem \cite[Theorem 4.2]{[IY]}. As an application, we show that the quotient category of a Frobenius $n$-angulated category has a natural structure of $n$-angulated category, which is a higher version of \cite[Theorem 7.2]{[B]}. We can compare it with \cite[Theorem 5.11]{[J]}.

This paper is organized as follows. In Section 2, we recall the definition of an $n$-angulated category and give some useful facts. In Section 3, we define mutation pair in an $n$-angulated category, then state and prove our main results.

\section{N-angulated categories}

In this section we recall some basics on $n$-angulated categories from \cite{[GKO]} and \cite{[BT1]}.

Let $\mathcal{C}$ be an additive category equipped with an automorphism $\Sigma:\mathcal{C}\rightarrow\mathcal{C}$, and $n$  an integer greater than or equal to three. An $n$-$\Sigma$-$sequence$ in $\mathcal{C}$ is a sequence of morphisms
$$X_1\xrightarrow{f_1}X_2\xrightarrow{f_2}X_3\xrightarrow{f_3}\cdots\xrightarrow{f_{n-1}}X_n\xrightarrow{f_n}\Sigma X_1.$$
Its $left\ rotation$ is the $n$-$\Sigma$-sequence
$$X_2\xrightarrow{f_2}X_3\xrightarrow{f_3}X_4\xrightarrow{f_4}\cdots\xrightarrow{f_{n-1}}X_n\xrightarrow{f_n}\Sigma X_1\xrightarrow{(-1)^n\Sigma f_1}\Sigma X_2.$$  We can define $right\ rotation\ of\ an$ $n$-$\Sigma$-$sequence$ similarly.
A $morphism\ of$ $n$-$\Sigma$-$sequences$ is  a sequence of morphisms $\varphi=(\varphi_1,\varphi_2,\cdots,\varphi_n)$ such that the following diagram commutes
$$\xymatrix{
X_1 \ar[r]^{f_1}\ar[d]^{\varphi_1} & X_2 \ar[r]^{f_2}\ar[d]^{\varphi_2} & X_3 \ar[r]^{f_3}\ar[d]^{\varphi_3} & \cdots \ar[r]^{f_{n-1}}& X_n \ar[r]^{f_n}\ar[d]^{\varphi_n} & \Sigma X_1 \ar[d]^{\Sigma \varphi_1}\\
Y_1 \ar[r]^{g_1} & Y_2 \ar[r]^{g_2} & Y_3 \ar[r]^{g_3} & \cdots \ar[r]^{g_{n-1}} & Y_n \ar[r]^{g_n}& \Sigma Y_1\\
}$$
where each row is an $n$-$\Sigma$-sequence. It is an {\em isomorphism} if $\varphi_1, \varphi_2, \cdots, \varphi_n$ are all isomorphisms in $\mathcal{C}$.

\begin{defn} (\cite{[GKO]})
An $n$-$angulated\ category$ is a triple $(\mathcal{C}, \Sigma, \Theta)$, where $\mathcal{C}$ is an additive category, $\Sigma$ is an automorphism of $\mathcal{C}$, and $\Theta$ is a class of $n$-$\Sigma$-sequences (whose elements are called $n$-angles), which satisfies the following axioms:

(N1) (a) The class $\Theta$ is closed under direct sums and direct summands.

(b) For each object $X\in\mathcal{C}$ the trivial sequence
$$X\xrightarrow{1_X}X\rightarrow 0\rightarrow\cdots\rightarrow 0\rightarrow \Sigma X$$
belongs to $\Theta$.

(c) For each morphism $f_1:X_1\rightarrow X_2$ in $\mathcal{C}$, there exists an $n$-angle whose first morphism is $f_1$.
\end{defn}

(N2) An $n$-$\Sigma$-sequence belongs to $\Theta$ if and only if its left rotation belongs to $\Theta$.

(N3) Each commutative diagram
$$\xymatrix{
X_1 \ar[r]^{f_1}\ar[d]^{\varphi_1} & X_2 \ar[r]^{f_2}\ar[d]^{\varphi_2} & X_3 \ar[r]^{f_3}\ar@{-->}[d]^{\varphi_3} & \cdots \ar[r]^{f_{n-1}}& X_n \ar[r]^{f_n}\ar@{-->}[d]^{\varphi_n} & \Sigma X_1 \ar[d]^{\Sigma \varphi_1}\\
Y_1 \ar[r]^{g_1} & Y_2 \ar[r]^{g_2} & Y_3 \ar[r]^{g_3} & \cdots \ar[r]^{g_{n-1}} & Y_n \ar[r]^{g_n}& \Sigma Y_1\\
}$$ with rows in $\Theta$ can be completed to a morphism of  $n$-$\Sigma$-sequences.

(N4) In the situation of (N3), the morphisms $\varphi_3, \varphi_4, \cdots,\varphi_n$ can be chosen such that the mapping cone
$$X_2\oplus Y_1\xrightarrow{\left(
                              \begin{smallmatrix}
                                -f_2 & 0 \\
                                \varphi_2 & g_1 \\
                              \end{smallmatrix}
                            \right)}
 X_3\oplus Y_2 \xrightarrow{\left(
                              \begin{smallmatrix}
                                -f_3 & 0 \\
                                \varphi_3 & g_2 \\
                              \end{smallmatrix}
                            \right)}
 \cdots \xrightarrow{\left(
                            \begin{smallmatrix}
                               -f_n & 0 \\
                                \varphi_n & g_{n-1} \\
                             \end{smallmatrix}
                           \right)}
 \Sigma X_1\oplus Y_n \xrightarrow{\left(
                              \begin{smallmatrix}
                                -\Sigma f_1 & 0 \\
                                \Sigma\varphi_1 & g_n \\
                              \end{smallmatrix}
                            \right)}
 \Sigma X_2\oplus \Sigma Y_1 \\
$$
belongs to $\Theta$.

\vspace{2mm}
 We recall a higher $~``$octahedral axiom" (N4$'$) for an $n$-angulated category as follows, which was introduced by Bergh and Thaule \cite{[BT1]}.

(N4$'$) Given a commutative diagram
$$\xymatrix{
X_1\ar[r]^{f_1}\ar@{=}[d] & X_2 \ar[r]^{f_2}\ar[d]^{\varphi_2} & X_3 \ar[r]^{f_3} & \cdots\ar[r]^{f_{n-2}} & X_{n-1}\ar[r]^{f_{n-1}} & X_n \ar[r]^{f_n} & \Sigma X_1\ar@{=}[d] \\
X_1\ar[r]^{g_1} & Y_2 \ar[r]^{g_2}\ar[d]^{h_2} & Y_3\ar[r]^{g_3} & \cdots\ar[r]^{g_{n-2}} & Y_{n-1}\ar[r]^{g_{n-1}} & Y_n \ar[r]^{g_n} & \Sigma X_1 \\
& Z_3\ar[d]^{h_3} & & & & & \\
& \vdots\ar[d]^{h_{n-2}} & & & & & \\
& Z_{n-1}\ar[d]^{h_{n-1}} & & & & & \\
& Z_n\ar[d]^{h_n} & & & & & \\
& \Sigma X_2 & & & & & \\
}$$ whose top rows and second column are $n$-angles. Then there exist morphisms $\varphi_i:X_i\rightarrow Y_i (i=3,4,\cdots,n)$, $\psi_j: Y_j\rightarrow Z_j (j=3,4,\cdots,n)$ and $\phi_k: X_k\rightarrow Z_{k-1} (k=4,5,\cdots,n)$ with the following two properties:

(a) The sequence $(1_{X_1},\varphi_2, \varphi_3,\cdots,\varphi_n)$ is a morphism of $n$-angles.

(b) The $n$-$\Sigma$-sequence
$$X_3\xrightarrow{\left(
                    \begin{smallmatrix}
                      f_3 \\
                      \varphi_3 \\
                    \end{smallmatrix}
                  \right)} X_4\oplus Y_3\xrightarrow{\left(
                             \begin{smallmatrix}
                               -f_4 & 0 \\
                               \varphi_4 & -g_3 \\
                               \phi_4 & \psi_3 \\
                             \end{smallmatrix}
                           \right)}
 X_5\oplus Y_4\oplus Z_3\xrightarrow{\left(
                                       \begin{smallmatrix}
                                         -f_5 & 0 & 0 \\
                                         -\varphi_5 & -g_4 & 0 \\
                                         \phi_5 & \psi_4 & h_3 \\
                                       \end{smallmatrix}
                                     \right)}X_6\oplus Y_5\oplus Z_4$$
$$\xrightarrow{\left(
                                       \begin{smallmatrix}
                                         -f_6 & 0 & 0 \\
                                         \varphi_6 & -g_5 & 0 \\
                                         \phi_6 & \psi_5 & h_4 \\
                                       \end{smallmatrix}
                                     \right)}\cdots\xrightarrow{\scriptsize\left(\begin{smallmatrix}
             -f_{n-1} & 0 & 0 \\
             (-1)^{n-1}\varphi_{n-1} & -g_{n-2} & 0 \\
             \phi_{n-1} & \psi_{n-2} & h_{n-3} \\
             \end{smallmatrix}
             \right)}X_n\oplus Y_{n-1}\oplus Z_{n-2}$$
$$\xrightarrow{\left(
                                                       \begin{smallmatrix}
                                                         (-1)^n\varphi_n &-g_{n-1} &0 \\
                                                          \phi_n& \psi_{n-1}& h_{n-2} \\
                                                       \end{smallmatrix}
                                                     \right)}Y_n\oplus Z_{n-1}\xrightarrow{(\psi_n\ h_{n-1})}Z_n\xrightarrow{\Sigma f_2\cdot h_n}\Sigma X_3 \hspace{10mm}(2.1)$$
is an $n$-angle, and $h_n\cdot\psi_n=\Sigma f_1\cdot g_n$.

\begin{thm} (\cite[Theorem 4.4]{[BT1]}) \label{2.3}
If $\Theta$ is a class of $n$-$\Sigma$-sequences satisfying the axioms (N1), (N2) and (N3), then $\Theta$ satisfies (N4) if and only if $\Theta$ satisfies (N4$'$).
\end{thm}

At the end of this section, we give two useful facts on $n$-angulated categories.

\begin{lem}\label{2.1} (\cite[Remarks 1.2 (c)]{[GKO]})
If $(\mathcal{C},\Sigma,\Theta)$ is an $n$-angulated category, then the opposite category $(\mathcal{C}^{op},\Sigma^{-1},\Phi)$ is also an $n$-angulated category, where $$\Sigma^{-1}X_n\xleftarrow{(-1)^n\Sigma^{-1}f_n}X_{1}\xleftarrow{f_1}X_2\xleftarrow{f_2} \cdots \xleftarrow{f_{n-2}}X_{n-1}\xleftarrow{f_{n-1}}X_n$$ is an $n$-angle in $\Phi$ provided that $$X_1\xrightarrow{f_1}X_2\xrightarrow{f_2}X_3\xrightarrow{f_3}\cdots\xrightarrow{f_{n-1}}X_n\xrightarrow{f_n}\Sigma X_1$$ is an $n$-angle in $\Theta$.
\end{lem}

\begin{lem}\label{2.2}
Let $$X_1\xrightarrow{f_1}X_2\xrightarrow{f_2}X_3\xrightarrow{f_3}\cdots\xrightarrow{f_{n-1}}X_n\xrightarrow{f_n}\Sigma X_1$$ be an $n$-angle in an $n$-angulated category $\mathcal{C}$. Then

(1) the induced sequence
$$\cdots\rightarrow\mathcal{C}(-,X_1)\rightarrow\mathcal{C}(-,X_2)\rightarrow\cdots\rightarrow\mathcal{C}(-,X_n)\rightarrow\mathcal{C}(-,\Sigma X_1)\rightarrow\cdots$$ is exact;

(2) the induced sequence
$$\cdots\rightarrow\mathcal{C}(\Sigma X_1,-)\rightarrow\mathcal{C}(X_n,-)\rightarrow\cdots\rightarrow\mathcal{C}(X_2,-)\rightarrow\mathcal{C}( X_1,-)\rightarrow\cdots$$ is exact.
\end{lem}

\begin{proof}
(1) We refer to \cite[Proposition 1.5 (a)]{[GKO]}. (2) Follows from (1) and Lemma \ref{2.1}.
\end{proof}

\section{Main results}

Let $\mathcal{C}$ be an additive category and $\mathcal{D}$ a subcategory of $\mathcal{C}$. When
we say that $\mathcal{D}$ is a subcategory of $\mathcal{C}$, we always mean that $\mathcal{D}$ is full and is closed under
isomorphisms, direct sums and direct summands. A morphism $f:X\rightarrow Y$ in $\mathcal{C}$ is called $\mathcal{D}$-$monic$ if $\mathcal{C}(Y,D)\xrightarrow{\mathcal{C}(f,D)} \mathcal{C}(X,D)\rightarrow 0$ is exact for any object $D\in\mathcal{D}$. A morphism $f:X\rightarrow D$ in $\mathcal{C}$ is called a $left$ $\mathcal{D}$-$approximation$ $of$ $X$ if $f$ is $\mathcal{D}$-monic and $D\in\mathcal{D}$. We can defined $\mathcal{D}$-$epic$ morphism and $right$ $\mathcal{D}$-$approximation$ dually.

\begin{defn}
Let $\mathcal{C}$ be an $n$-angulated category, and $\mathcal{D}\subseteq\mathcal{Z}$ be subcategories of $\mathcal{C}$. The pair $(\mathcal{Z},\mathcal{Z})$ is called a $\mathcal{D}$-$mutation\ pair$ if it satisfies:

(1) For any object $X\in\mathcal{Z}$, there exists an $n$-angle $$X\xrightarrow{d_1}D_1\xrightarrow{d_2}D_2\xrightarrow{d_3}\cdots\xrightarrow{d_{n-2}}D_{n-2}\xrightarrow{d_{n-1}}Y\xrightarrow{d_n}\Sigma X$$
where $D_i\in\mathcal{D}, Y\in\mathcal{Z}, d_1$ is a left $\mathcal{D}$-approximation and $d_{n-1}$ is a right $\mathcal{D}$-approximation.

(2) For any object $Y\in\mathcal{Z}$, there exists an $n$-angle $$X\xrightarrow{d_1}D_1\xrightarrow{d_2}D_2\xrightarrow{d_3}\cdots\xrightarrow{d_{n-2}}D_{n-2}\xrightarrow{d_{n-1}}Y\xrightarrow{d_n}\Sigma X$$
where $X\in\mathcal{Z}, D_i\in\mathcal{D}, d_1$ is a left $\mathcal{D}$-approximation and $d_{n-1}$ is a right $\mathcal{D}$-approximation.
\end{defn}

Note that if $\mathcal{C}$ is a triangulated category, i.e. when $n=3$, and $\mathcal{D}$ is a rigid subcategory, then a $\mathcal{D}$-mutation pair is just the same as Iyama-Yoshino's definition \cite[Definition 2.5]{[IY]}.

\begin{example}
We first recall the standard construction of $n$-angulated categories given by Geiss-Keller-Oppermann \cite[Theorem 1]{[GKO]}.
Let $\mathcal{C}$ be a triangulated category and $\mathcal{T}$ an $(n-2)$-cluster tilting subcategory which is closed under $\Sigma^{n-2}$, where $\Sigma$ is the suspension functor of $\mathcal{C}$. Then $(\mathcal{T},\Sigma^{n-2},\Theta)$ is an $n$-angulated category, where $\Theta$ is the class of all sequences
$$X_1\xrightarrow{f_1}X_2\xrightarrow{f_2}X_3\xrightarrow{f_3}\cdots\xrightarrow{f_{n-1}}X_n\xrightarrow{f_n}\Sigma^{n-2} X_1$$
such that there exists a diagram
 $$\xymatrixcolsep{0.3pc}
 \xymatrix{
                & X_2 \ar[dr]\ar[rr]^{f_2}  &  & X_3  \ar[dr]  & & \cdots  & & X_{n-1} \ar[dr]^{f_{n-1}}      \\
X_1 \ar[ur]^{f_1} & \mid & \ar[ll]  X_{2.5}\ar[ur] & \mid &  \ar[ll]  X_{3.5} & \cdots & X_{n-1.5}\ar[ur] & \mid & \ar[ll] X_n   }$$
with $X_i\in\mathcal{T}$ for all $i\in\mathbb{Z}$, such that all oriented triangles are triangles in $\mathcal{C}$, all non-oriented triangles commute, and $f_n$ is the composition along the lower edge of the diagram.

If $\mathcal{D}\subseteq\mathcal{Z}$ are subcategories of $\mathcal{T}$ and $(\mathcal{Z},\mathcal{Z})$ is a $\mathcal{D}$-mutation pair in triangulated category $\mathcal{C}$, then it is easy to see that $(\mathcal{Z},\mathcal{Z})$ is a $\mathcal{D}$-mutation pair in $n$-angulated category $\mathcal{T}$.
\end{example}

For a $\mathcal{D}$-mutation pair $(\mathcal{Z},\mathcal{Z})$ in an $n$-angulated category $\mathcal{C}$, consider the quotient category $\mathcal{Z}/\mathcal{D}$, whose objects are objects of $\mathcal{Z}$ and given two objects $X, Y$, the set of morphisms $(\mathcal{Z}/\mathcal{D})(X,Y)$ is defined as the quotient group $\mathcal{Z}(X,Y)/[\mathcal{D}](X,Y)$, where $[\mathcal{D}](X,Y)$ is the subgroup of morphisms from $X$ to $Y$ factoring through some object in $\mathcal{D}$. For any morphism $f:X\rightarrow Y$ in $\mathcal{Z}$, we denote by $\underline{f}$ the image of $f$ under the quotient functor $\mathcal{Z}\rightarrow\mathcal{Z}/\mathcal{D}$.

\begin{lem}\label{3.1}
Let $$\xymatrix{
X_1 \ar[r]^{f_1}\ar[d]^{a_1} & X_2 \ar[r]^{f_2}\ar[d]^{a_2} & X_3 \ar[r]^{f_3}\ar[d]^{a_3} & \cdots \ar[r]^{f_{n-2}} & X_{n-1} \ar[r]^{f_{n-1}}\ar[d]^{a_{n-1}}& X_n \ar[r]^{f_n}\ar[d]^{a_n} & \Sigma X_1 \ar[d]^{\Sigma a_1}\\
X \ar[r]^{d_1} & D_1 \ar[r]^{d_2} & D_2 \ar[r]^{d_3} & \cdots \ar[r]^{d_{n-2}} & D_{n-2} \ar[r]^{d_{n-1}} & Y \ar[r]^{d_n} & \Sigma X,\\
}$$
$$\xymatrix{
X_1 \ar[r]^{f_1}\ar[d]^{a_1} & X_2 \ar[r]^{f_2}\ar[d]^{a_2'} & X_3 \ar[r]^{f_3}\ar[d]^{a_3'} & \cdots \ar[r]^{f_{n-2}} & X_{n-1} \ar[r]^{f_{n-1}}\ar[d]^{a_{n-1}'}& X_n \ar[r]^{f_n}\ar[d]^{a_n'} & \Sigma X_1 \ar[d]^{\Sigma a_1}\\
X \ar[r]^{d_1} & D_1 \ar[r]^{d_2} & D_2 \ar[r]^{d_3} & \cdots \ar[r]^{d_{n-2}} & D_{n-2} \ar[r]^{d_{n-1}} & Y \ar[r]^{d_n} & \Sigma X \\
}$$ be morphisms of $n$-angles, where $X,Y,X_i\in\mathcal{Z}$ and $D_j\in\mathcal{D}$. Then $\underline{a_n}=\underline{a_n'}$ in quotient category $\mathcal{Z}/\mathcal{D}$.
\end{lem}

\begin{proof}
Since $d_na_n=\Sigma a_1\cdot f_n=d_na_n'$, we have $d_n(a_n-a_n')=0$. By Lemma \ref{2.2}(1) we obtain that $a_n-a_n'$ factors through $d_{n-1}$ thus $\underline{a_n}=\underline{a_n'}$.
\end{proof}

Now we construct a functor $T: \mathcal{Z}/\mathcal{D}\rightarrow\mathcal{Z}/\mathcal{D}$ as follows.
For any object $X\in\mathcal{Z}$, fix an $n$-angle $$X\xrightarrow{d_1}D_1\xrightarrow{d_2}D_2\xrightarrow{d_3}\cdots\xrightarrow{d_{n-2}}D_{n-2}\xrightarrow{d_{n-1}}TX\xrightarrow{d_n}\Sigma X$$
with $D_i\in\mathcal{D}, TX\in\mathcal{Z}, d_1$ is a left $\mathcal{D}$-approximation  and $d_{n-1}$ is a right $\mathcal{D}$-approximation. For any morphism $f\in\mathcal{Z}(X, X')$, since $d_1$ is a left $\mathcal{D}$-approximation, there exist morphisms $a_i$ and $g$ which make the following diagram commutative.
$$\xymatrix{
X \ar[r]^{d_1}\ar[d]^{f} & D_1 \ar[r]^{d_2}\ar[d]^{a_1} & D_2 \ar[r]^{d_3}\ar[d]^{a_2} & \cdots \ar[r]^{d_{n-2}} & D_{n-2} \ar[r]^{d_{n-1}}\ar[d]^{a_{n-2}}& TX \ar[r]^{d_n}\ar[d]^{g} & \Sigma X \ar[d]^{\Sigma f}\\
X' \ar[r]^{d_1'} & D_1' \ar[r]^{d_2'} & D_2' \ar[r]^{d_3'} & \cdots \ar[r]^{d_{n-2}'} & D_{n-2}'  \ar[r]^{d_{n-1}'} & TX' \ar[r]^{d_n'}& \Sigma X'\\
}$$
Now we put $T\underline{f}=\underline{g}$. Note that $\Sigma f\cdot d_n=d_n'\cdot g$, which will be used frequently.

\begin{prop}
The functor $T: \mathcal{Z}/\mathcal{D}\rightarrow\mathcal{Z}/\mathcal{D}$ is  a well defined equivalence.
\end{prop}

\begin{proof}
By Lemma \ref{3.1}, it is easy to see that $T$ is a well defined functor. We can construct another functor $T':\mathcal{Z}/\mathcal{D}\rightarrow\mathcal{Z}/\mathcal{D}$ by a dual manner. For any object $X\in\mathcal{Z}$, fix an $n$-angle $$T'X\xrightarrow{d_1}D_1\xrightarrow{d_2}D_2\xrightarrow{d_3}\cdots\xrightarrow{d_{n-2}}D_{n-2}\xrightarrow{d_{n-1}}X\xrightarrow{d_n}\Sigma T'X$$
with $D_i\in\mathcal{D}, T'X\in\mathcal{Z}, d_1$ is a left $\mathcal{D}$-approximation  and $d_{n-1}$ is a right $\mathcal{D}$-approximation. For any morphism $f\in\mathcal{Z}(X, X')$, since $d_{n-1}$ is a right $\mathcal{D}$-approximation, there exist morphisms $b_i$ and $g$ which make the following diagram commutative.
$$\xymatrix{
T'X \ar[r]^{d_1}\ar[d]^{g} & D_1 \ar[r]^{d_2}\ar[d]^{b_{n-2}} & D_2 \ar[r]^{d_3}\ar[d]^{b_{n-3}} & \cdots \ar[r]^{d_{n-2}} & D_{n-2} \ar[r]^{d_{n-1}}\ar[d]^{b_{1}}& X \ar[r]^{d_n}\ar[d]^{f} & \Sigma T'X \ar[d]^{\Sigma g}\\
T'X' \ar[r]^{d_1'} & D_1' \ar[r]^{d_2'} & D_2' \ar[r]^{d_3'} & \cdots \ar[r]^{d_{n-2}'} & D_{n-2}'  \ar[r]^{d_{n-1}'} & X' \ar[r]^{d_n'}& \Sigma T'X'\\
}$$
We put $T'\underline{f}=\underline{g}$. The dual of Lemma \ref{3.1} implies that $T'$ is a well defined functor.
It is easy to check that $T'$ gives a quasi-inverse of $T$.
\end{proof}

\begin{defn}\label{3.3}
Let $$X_1\xrightarrow{f_1}X_2\xrightarrow{f_2}X_3\xrightarrow{f_3}\cdots\xrightarrow{f_{n-1}}X_n\xrightarrow{f_n}\Sigma X_1$$ be an $n$-angle, where $X_i\in\mathcal{Z}$ and $f_1$ is $\mathcal{D}$-monic. Then there exists a commutative diagram of $n$-angles
$$\xymatrix{
X_1 \ar[r]^{f_1}\ar@{=}[d] & X_2 \ar[r]^{f_2}\ar[d]^{a_2} & X_3 \ar[r]^{f_3}\ar[d]^{a_3} & \cdots \ar[r]^{f_{n-2}} & X_{n-1} \ar[r]^{f_{n-1}}\ar[d]^{a_{n-1}}& X_n \ar[r]^{f_n}\ar[d]^{a_n} & \Sigma X_1 \ar@{=}[d]^{\ \ \ \mbox{(3.1)}} \\
X_1 \ar[r]^{d_1} & D_1 \ar[r]^{d_2} & D_2 \ar[r]^{d_3} & \cdots \ar[r]^{d_{n-2}} & D_{n-2} \ar[r]^{d_{n-1}}& TX_1 \ar[r]^{d_n} & \Sigma X_1. \\
}$$
The $n$-$T$-sequence $$X_1\xrightarrow{\underline{f_1}}X_2\xrightarrow{\underline{f_2}}X_3\xrightarrow{\underline{f_3}}\cdots\xrightarrow{\underline{f_{n-1}}}
X_n\xrightarrow{\underline{a_n}} TX_1$$ is called a $standard\ n$-$angle$ in $\mathcal{Z}/\mathcal{D}$ . We define $\Phi$ the class of $n$-$T$-sequences which are isomorphic to standard $n$-angles.
\end{defn}

\begin{lem}\label{3.2}
Assume that we have a commutative diagram
$$\xymatrix{
X_1 \ar[r]^{f_1}\ar[d]^{\varphi_1} & X_2 \ar[r]^{f_2}\ar[d]^{\varphi_2} & X_3 \ar[r]^{f_3}\ar[d]^{\varphi_3} & \cdots \ar[r]^{f_{n-1}}& X_n \ar[r]^{f_n}\ar[d]^{\varphi_n} & \Sigma X_1 \ar[d]^{\Sigma \varphi_1}\\
Y_1 \ar[r]^{g_1} & Y_2 \ar[r]^{g_2} & Y_3 \ar[r]^{g_3} & \cdots \ar[r]^{g_{n-1}} & Y_n \ar[r]^{g_n}& \Sigma Y_1\\
}$$
where the rows are $n$-angles in $\mathcal{C}$,  all $X_i,Y_i\in \mathcal{Z}$ and $f_1, g_1$ are $\mathcal{D}$-monic. Then we have the following commutative diagram
$$\xymatrix{
X_1 \ar[r]^{\underline{f_1}}\ar[d]^{\underline{\varphi_1}} & X_2 \ar[r]^{\underline{f_2}}\ar[d]^{\underline{\varphi_2}} & X_3 \ar[r]^{\underline{f_3}}\ar[d]^{\underline{\varphi_3}} & \cdots \ar[r]^{\underline{f_{n-1}}}& X_n \ar[r]^{\underline{a_n}}\ar[d]^{\underline{\varphi_n}} & TX_1 \ar[d]^{T\underline{\varphi_1}}\\
Y_1 \ar[r]^{\underline{g_1}} & Y_2 \ar[r]^{\underline{g_2}} & Y_3 \ar[r]^{\underline{g_3}} & \cdots \ar[r]^{\underline{g_{n-1}}} & Y_n \ar[r]^{\underline{b_n}}& TY_1\\
}$$ where the rows are standard $n$-angles in $\mathcal{Z}/\mathcal{D}$.
\end{lem}

\begin{proof}
We only need to show that $T\underline{\varphi_1}\cdot \underline{a_n}=\underline{b_n}\cdot\underline{\varphi_n}$. By the constructions of  the morphism $T\underline{\varphi_1}$ and the standard $n$-angles in $\mathcal{Z}/\mathcal{D}$, we have the following two commutative diagrams of $n$-angles
$$\xymatrix{
X_1 \ar[r]^{f_1}\ar@{=}[d] & X_2 \ar[r]^{f_2}\ar[d]^{a_2} & X_3 \ar[r]^{f_3}\ar[d]^{a_3} & \cdots \ar[r]^{f_{n-2}} & X_{n-1} \ar[r]^{f_{n-1}}\ar[d]^{a_{n-1}}& X_n \ar[r]^{f_n}\ar[d]^{a_n} & \Sigma X_1 \ar@{=}[d]\\
X_1 \ar[r]^{d_1}\ar[d]^{\varphi_1} & D_1 \ar[r]^{d_2}\ar[d]^{a_1'} & D_2 \ar[r]^{d_3}\ar[d]^{a_2'} & \cdots \ar[r]^{d_{n-2}} & D_{n-2} \ar[r]^{d_{n-1}}\ar[d]^{a_{n-2}'}& TX_1 \ar[r]^{d_n}\ar[d]^{\psi_1} & \Sigma X_1\ar[d]^{\Sigma \varphi_1}\\
Y_1 \ar[r]^{d_1'} & D_1' \ar[r]^{d_2'} & D_2' \ar[r]^{d_3'} & \cdots \ar[r]^{d_{n-2}'} & D_{n-2}'  \ar[r]^{d_{n-1}'} & TY_1 \ar[r]^{d_n'}& \Sigma Y_1,\\
}$$
$$\xymatrix{
X_1 \ar[r]^{f_1}\ar[d]^{\varphi_1} & X_2 \ar[r]^{f_2}\ar[d]^{\varphi_2} & X_3 \ar[r]^{f_3}\ar[d]^{\varphi_3} & \cdots\ar[r]^{f_{n-2}} & X_{n-1}  \ar[r]^{f_{n-1}}\ar[d]^{\varphi_{n-1}}& X_n \ar[r]^{f_n}\ar[d]^{\varphi_n} & \Sigma X_1 \ar[d]^{\Sigma \varphi_1}\\
Y_1 \ar[r]^{g_1}\ar@{=}[d] & Y_2 \ar[r]^{g_2}\ar[d]^{b_2} & Y_3 \ar[r]^{g_3}\ar[d]^{b_3} & \cdots \ar[r]^{g_{n-2}} & Y_{n-1} \ar[r]^{g_{n-1}}\ar[d]^{b_{n-1}} & Y_n \ar[r]^{g_n}\ar[d]^{b_n}& \Sigma Y_1\ar@{=}[d]\\
Y_1 \ar[r]^{d_1'} & D_1' \ar[r]^{d_2'} & D_2' \ar[r]^{d_3'} & \cdots \ar[r]^{d_{n-2}'} & D_{n-2}'  \ar[r]^{d_{n-1}'} & TY_1 \ar[r]^{d_n'}& \Sigma Y_1
}$$
where $T\underline{\varphi_1}=\underline{\psi_1}$. Lemma \ref{3.1}  implies that $T\underline{\varphi_1}\cdot \underline{a_n}=\underline{b_n}\cdot\underline{\varphi_n}$.
\end{proof}

\begin{defn}
Let $\mathcal{C}$ be an $n$-angulated category. A subcategory $\mathcal{Z}$ is called extension-closed if for any $n$-angle
$$X_1\xrightarrow{f_1}X_2\xrightarrow{f_2}X_3\xrightarrow{f_3}\cdots\xrightarrow{f_{n-1}}X_n\xrightarrow{f_n}\Sigma X_1$$ in $\mathcal{C}$, the objects $X_1,X_n\in\mathcal{Z}$ implies that $X_2,X_3,\cdots, X_{n-1}\in\mathcal{Z}$.
\end{defn}

Now we can state and prove our main theorem.

\begin{thm} \label{3.4}
Let $\mathcal{C}$ be an $n$-angulated category, and $\mathcal{D}\subseteq\mathcal{Z}$ be subcategories of $\mathcal{C}$. If $(\mathcal{Z},\mathcal{Z})$ is a $\mathcal{D}$-mutation pair and $\mathcal{Z}$ is extension-closed, then the quotient category $\mathcal{Z}/\mathcal{D}$ is an $n$-angulated category with respect to the autoequivalence $T$ and $n$-angles defined in Definition \ref{3.3}.
\end{thm}

\begin{proof}
We will check that the class of $n$-$T$-sequences $\Phi$, which is defined in Definition \ref{3.3}, satisfies the axioms (N1), (N2), (N3) and (N4$'$). It is easy to see from the definition that (N1)(a) is satisfied.

For any object $X\in\mathcal{Z}$, the identity  morphism of $X$ is $\mathcal{D}$-monic. The commutative diagram
$$\xymatrix{
X \ar[r]^{1_X}\ar@{=}[d] & X \ar[r]^{0}\ar[d]^{d_1} & 0 \ar[r]^{0}\ar[d]^{0} & \cdots \ar[r]^{0} & 0 \ar[r]^{0}\ar[d]^{0}& 0 \ar[r]^{0}\ar[d]^{0} & \Sigma X \ar@{=}[d]\\
X \ar[r]^{d_1} & D_1 \ar[r]^{d_2} & D_2 \ar[r]^{d_3} & \cdots \ar[r]^{d_{n-2}} & D_{n-2} \ar[r]^{d_{n-1}}& TX \ar[r]^{d_n} & \Sigma X \\
}$$ shows that $$X\xrightarrow{1_X}X\rightarrow 0\rightarrow\cdots\rightarrow 0\rightarrow TX$$ belongs to $\Phi$. Thus (N1)(b) is satisfied.

For any morphism $f:X\rightarrow Y$ in $\mathcal{Z}$, the morphism $\left(
                                                          \begin{smallmatrix}
                                                            f \\
                                                            d_1 \\
                                                          \end{smallmatrix}
                                                        \right)
:X\rightarrow Y\oplus D_1$ is $\mathcal{D}$-monic, where $d_1$ is a left $\mathcal{D}$-approximation. Now we have three $n$-angles in $\mathcal{C}$:
$$ X\xrightarrow{\left(
                                                          \begin{smallmatrix}
                                                            f \\
                                                            d_1 \\
                                                          \end{smallmatrix}
                                                        \right)}
Y\oplus D_1\xrightarrow{f_2}X_3\xrightarrow{f_3}\cdots\xrightarrow{f_{n-2}}X_{n-1}\xrightarrow{f_{n-1}}X_n\xrightarrow{f_n}\Sigma X \hspace{5mm} (3.2)$$
$$Y\oplus D_1\xrightarrow{\left(
                            \begin{smallmatrix}
                              0 & 1_{D_1} \\
                            \end{smallmatrix}
                          \right)
}D_1\xrightarrow{0}0\xrightarrow{0}\cdots\xrightarrow{0}0\xrightarrow{0}\Sigma Y\xrightarrow{\left(
                                                                                               \begin{smallmatrix}
                                                                                                 1_{\Sigma Y} \\
                                                                                                 0 \\
                                                                                               \end{smallmatrix}
                                                                                             \right)
}\Sigma Y\oplus\Sigma D_1$$
$$X\xrightarrow{d_1}D_1\xrightarrow{d_2}D_2\xrightarrow{d_3}\cdots\xrightarrow{d_{n-2}}D_{n-2}\xrightarrow{d_{n-1}}TX\xrightarrow{d_n}\Sigma X$$
where the first $n$-angle is by (N1)(c). Since $d_1=\left(
                            \begin{smallmatrix}
                              0 & 1_{D_1} \\
                            \end{smallmatrix}
                          \right)
                          \left(
                                                          \begin{smallmatrix}
                                                            f \\
                                                            d_1 \\
                                                          \end{smallmatrix}
                                                        \right)$,
we use axiom (N4$'$) to get an $n$-angle in $\mathcal{C}$
$$X_3\rightarrow X_4\oplus D_2\rightarrow X_5\oplus D_3\rightarrow\cdots\rightarrow X_n\oplus D_{n-2}\rightarrow TX\rightarrow \Sigma Y\rightarrow \Sigma X_3.$$ Thus by (N2) we obtain the following $n$-angle
$$Y\rightarrow X_3\rightarrow X_4\oplus D_2\rightarrow X_5\oplus D_3\rightarrow\cdots\rightarrow X_n\oplus D_{n-2}\rightarrow TX\rightarrow \Sigma Y.$$ Since $\mathcal{Z}$ is extension-closed and $Y, TX\in\mathcal{Z}$, we have $X_3, X_4\oplus D_2,\cdots,X_n\oplus D_{n-2}\in\mathcal{Z}$. Thus $X_i\in\mathcal{Z}, i=3,4,\cdots,n$. Therefore, the $n$-angle $(3.2)$ induces the standard $n$-angle in $\mathcal{Z}/\mathcal{D}$
$$X\xrightarrow{\underline{f}}Y\xrightarrow{\underline{f_2}}X_3\xrightarrow{\underline{f_3}}\cdots\xrightarrow{\underline{f_{n-2}}}
X_{n-1}\xrightarrow{\underline{f_{n-1}}}X_n\xrightarrow{\underline{a_n}}TX.$$
Thus (N1)(c) is satisfied.

(N2)  We only consider the case of standard $n$-angle, since the general case can be easily deduced.  Let $$X_1\xrightarrow{\underline{f_1}}X_2\xrightarrow{\underline{f_2}}X_3\xrightarrow{\underline{f_3}}\cdots\xrightarrow{\underline{f_{n-1}}}
X_n\xrightarrow{\underline{a_n}} TX_1$$ be a standard $n$-angle induced by the commutative diagram (3.1).
We need to show that its left rotation belongs to $\Phi$.

Consider the three $n$-angles in $\mathcal{C}$:
$$ \Sigma^{-1}X_n\xrightarrow{(-1)^{n}\Sigma^{-1}f_n} X_1\xrightarrow{f_1}X_2\xrightarrow{f_2}X_3\xrightarrow{f_3}\cdots\xrightarrow{f_{n-3}}X_{n-2}\xrightarrow{f_{n-2}}X_{n-1}\xrightarrow{f_{n-1}}X_n$$
$$X_1\xrightarrow{d_1}D_1\xrightarrow{d_2}D_2\xrightarrow{d_3}D_3\xrightarrow{d_4}\cdots\xrightarrow{d_{n-2}}D_{n-2}\xrightarrow{d_{n-1}}
TX_1\xrightarrow{d_n}\Sigma X_1$$
$$\Sigma^{-1}X_n\xrightarrow{0}D_1\xrightarrow{1_{D_1}}D_1\xrightarrow{0}0\xrightarrow{0}\cdots\xrightarrow{0}0\xrightarrow{0}
X_n\xrightarrow{1_{X_n}}X_n$$
where the first one is by (N2). Since $d_1\cdot\Sigma^{-1}f_n=a_2f_1\cdot\Sigma^{-1}f_n=0$, we use (N4$'$) to get an $n$-angle in $\mathcal{C}$
$$X_2\xrightarrow{\left(
                    \begin{smallmatrix}
                      f_2 \\
                      \varphi \\
                    \end{smallmatrix}
                  \right)
} X_3\oplus D_1\xrightarrow{\left(
                              \begin{smallmatrix}
                                -f_3 & 0 \\
                                \phi_3 & \psi \\
                              \end{smallmatrix}
                            \right)
} X_4\oplus D_2\xrightarrow{\left(
                              \begin{smallmatrix}
                                -f_4 & 0 \\
                                \phi_4 & d_3 \\
                              \end{smallmatrix}
                            \right)
} \cdots\xrightarrow{\left(
                              \begin{smallmatrix}
                                -f_{n-2} & 0 \\
                                \phi_{n-2} & d_{n-3} \\
                              \end{smallmatrix}
                            \right)
} $$$$ X_{n-1}\oplus D_{n-3}\xrightarrow{\left(
                              \begin{smallmatrix}
                                (-1)^{n}\varphi' & 0 \\
                                \phi_{n-1} & d_{n-2} \\
                              \end{smallmatrix}
                            \right)
} X_{n}\oplus D_{n-2}\xrightarrow{(\psi',d_{n-1})} TX_1\xrightarrow{\Sigma f_1\cdot d_n} \Sigma X_2$$
with $\varphi'=f_{n-1}$ and $d_n\cdot\psi'=(-1)^nf_n$. Note that $f_n=d_na_n$, we have $d_n(\psi'-(-1)^na_n)=0$. Thus $\psi'-(-1)^na_n$ factors through $d_{n-1}$, so that $\underline{\psi'}=(-1)^n\underline{a_n}$.
We claim that the morphism $\left(
                 \begin{smallmatrix}
                   f_2 \\
                   \varphi \\
                 \end{smallmatrix}
               \right): X_2\rightarrow X_3\oplus D_1
$ is $\mathcal{D}$-monic. In fact, for any morphism $f:X_2\rightarrow D$ with $D\in\mathcal{D}$, there is a morphism $g:D_1\rightarrow D$ such that $ff_1=gd_1=g\varphi f_1$ since $d_1$ is $\mathcal{D}$-monic. Now $(f-g\varphi)f_1=0$, there exists a morphism $h:X_3\rightarrow D$ with $f-g\varphi=hf_2$. So $f=(h\ g)\left(
                 \begin{smallmatrix}
                   f_2 \\
                   \varphi \\
                 \end{smallmatrix}
               \right)$.
Assume that $T\underline{f_1}=\underline{g_1}$, then we get $d_n'\cdot g_1=\Sigma f_1\cdot d_n$  by the definition of $T\underline{f_1}$.
Thus we can obtain the following commutative diagram.
$$\xymatrixcolsep{3pc}
\xymatrix{
X_2\ar[r]^{\left(
                    \begin{smallmatrix}
                      f_2 \\
                      \varphi \\
                    \end{smallmatrix}
                  \right ) \ \ \
}\ar@{=}[d] & X_3\oplus D_1\ar[r]^{\left(
                              \begin{smallmatrix}
                                -f_3 & 0 \\
                                \phi_3 & \psi \\
                              \end{smallmatrix}
                            \right)
} \ar@{-->}[d]^{b_1}  & \cdots \ar[r]^{\left(
                              \begin{smallmatrix}
                                (-1)^nf_{n-1} & 0 \\
                                \phi_{n-1} & d_{n-2} \\
                              \end{smallmatrix}
                            \right) \ \ \
}  & \ \ X_{n}\oplus D_{n-2}\ar[r]^{ \ \ \ \ (\psi'\ d_{n-1})} \ar@{-->}[d]^{b_{n-2}} & TX_1\ar[r]^{\Sigma f_1\cdot d_n}\ar[d]^{g_1} & \Sigma X_2 \ar@{=}[d]\\
X_2\ar[r]^{d_1'} & D_1'\ar[r]^{d_2'} & \cdots \ar[r]^{d_{n-2}'} & D_{n-2}'\ar[r]^{d_{n-1}'} & TX_2 \ar[r]^{d_n'} & \Sigma X_2
}$$
 It follows that
$$X_2\xrightarrow{\underline{f_2}}X_3\xrightarrow{-\underline{f_3}}X_4\xrightarrow{-\underline{f_4}}\cdots\xrightarrow{-\underline{f_{n-2}}}
X_{n-1}\xrightarrow{(-1)^n\underline{f_{n-1}}}X_{n}\xrightarrow{(-1)^n\underline{a_n}} TX_1\xrightarrow{T\underline{f_1}} TX_2 \ \ (3.3)$$
is a standard $n$-angle in $\mathcal{Z}/\mathcal{D}$. Thus $$X_2\xrightarrow{\underline{f_2}}X_3\xrightarrow{\underline{f_3}}X_4\xrightarrow{\underline{f_4}}\cdots\xrightarrow{\underline{f_{n-2}}}
X_{n-1}\xrightarrow{\underline{f_{n-1}}}X_{n}\xrightarrow{\underline{a_n}} TX_1\xrightarrow{(-1)^n T\underline{f_1}} TX_2 \ \ (3.4)$$
belongs to $\Phi$ since the $n$-$T$-sequences (3.3) and (3.4) are isomorphic.

On the contrary, let $$X_2\xrightarrow{\underline{f_2}}X_3\xrightarrow{\underline{f_3}}X_4\xrightarrow{\underline{f_4}}\cdots
\xrightarrow{\underline{f_{n-1}}}X_{n}\xrightarrow{\underline{f_n}} TX_1\xrightarrow{T\underline{f_1}} TX_2$$
be a standard $n$-angle in $\mathcal{Z}/\mathcal{D}$. By a similar way, we can show that its right rotation
is also an $n$-angle. We point out that we may use the dual of (N4$'$).

(N3) We only consider the case of standard $n$-angles. Suppose that there is a commutative diagram
$$\xymatrix{
X_1 \ar[r]^{\underline{f_1}}\ar[d]^{\underline{\varphi_1}} & X_2 \ar[r]^{\underline{f_2}}\ar[d]^{\underline{\varphi_2}} & X_3 \ar[r]^{\underline{f_3}}  & \cdots \ar[r]^{\underline{f_{n-1}}}& X_n \ar[r]^{\underline{a_n}} & TX_1 \ar[d]^{T\underline{\varphi_1} \ \ \ \mbox{(3.5)}}\\
Y_1 \ar[r]^{\underline{g_1}} & Y_2 \ar[r]^{\underline{g_2}} & Y_3 \ar[r]^{\underline{g_3}} & \cdots \ar[r]^{\underline{g_{n-1}}} & Y_n \ar[r]^{\underline{b_n}}& TY_1\\
}$$
with rows standard $n$-angles in $\mathcal{Z}/\mathcal{D}$. Since $\underline{\varphi_{2}}\cdot \underline{f_1}=\underline{g_1}\cdot\underline{\varphi_1}$ holds, $\varphi_2f_1-g_1\varphi_1$ factors through some object $D$ in $\mathcal{D}$. Assume that $\varphi_2f_1-g_1\varphi_1=gf$, where $f:X_1\rightarrow D$ and $g:D\rightarrow Y_2$. Since $f_1$ is $\mathcal{D}$-monic, there exists $h:X_2\rightarrow D$ such that $f=hf_1$. Note that $(\varphi_2-gh)f_1=g_1\varphi_1$, by (N3) we have the following commutative diagram
$$\xymatrix{
X_1 \ar[r]^{f_1}\ar[d]^{\varphi_1} & X_2 \ar[r]^{f_2}\ar[d]^{\varphi_2-gh} & X_3 \ar[r]^{f_3}\ar@{-->}[d]^{\varphi_3} & \cdots \ar[r]^{f_{n-1}}& X_n \ar[r]^{f_n}\ar@{-->}[d]^{\varphi_n} & \Sigma X_1 \ar[d]^{\Sigma \varphi_1}\\
Y_1 \ar[r]^{g_1} & Y_2 \ar[r]^{g_2} & Y_3 \ar[r]^{g_3} & \cdots \ar[r]^{g_{n-1}} & Y_n \ar[r]^{g_n}& \Sigma Y_1\\
}$$ with rows $n$-angles in $\mathcal{C}$. By Lemma \ref{3.2}, the diagram (3.5) can be completed to a morphism of $n$-angles.

(N4$'$) We only consider the case of standard $n$-angles. Let
$$X_1\xrightarrow{\underline{f_1}}X_2\xrightarrow{\underline{f_2}}X_3\xrightarrow{\underline{f_3}}\cdots\xrightarrow{\underline{f_{n-1}}}
X_n\xrightarrow{\underline{a_n}} TX_1$$
$$X_1\xrightarrow{\underline{\varphi_2 f_1}}Y_2\xrightarrow{\underline{g_2}}Y_3\xrightarrow{\underline{g_3}}\cdots\xrightarrow{\underline{g_{n-1}}}
Y_n\xrightarrow{\underline{b_n}} TX_1$$
$$X_2\xrightarrow{\underline{\varphi_2}}Y_2\xrightarrow{\underline{h_2}}Z_3\xrightarrow{\underline{h_3}}\cdots\xrightarrow{\underline{h_{n-1}}}
Z_n\xrightarrow{\underline{c_n}} TX_2$$
be three standard $n$-angles in $\mathcal{Z}/\mathcal{D}$ which are induced by the following three $n$-angles in $\mathcal{C}$ respectively:
$$X_1\xrightarrow{f_1}X_2\xrightarrow{f_2}X_3\xrightarrow{f_3}\cdots\xrightarrow{f_{n-1}}
X_n\xrightarrow{f_n} \Sigma X_1$$
$$X_1\xrightarrow{\varphi_2 f_1}Y_2\xrightarrow{g_2}Y_3\xrightarrow{g_3}\cdots\xrightarrow{g_{n-1}}
Y_n\xrightarrow{g_n} \Sigma X_1$$
$$X_2\xrightarrow{\varphi_2}Y_2\xrightarrow{h_2}Z_3\xrightarrow{h_3}\cdots\xrightarrow{h_{n-1}}
Z_n\xrightarrow{h_n} \Sigma X_2$$
where $f_1$ and $\varphi_2$ are $\mathcal{D}$-monic, thus $\varphi_2f_1$ is $\mathcal{D}$-monic too. Then we have $f_n=d_n a_n, g_n=d_n b_n$ and $h_n=d_n' c_n$ by the definition of standard $n$-angles in $\mathcal{Z}/\mathcal{D}$, where $d_n: TX_1\rightarrow \Sigma X_1$ and $d_n': TX_2\rightarrow \Sigma X_2$.

By the  axiom (N4$'$), there exist morphisms $\varphi_i:X_i\rightarrow Y_i (i=3,4,\cdots,n)$, $\psi_j: Y_j\rightarrow Z_j (j=3,4,\cdots,n)$ and $\phi_k: X_k\rightarrow Z_{k-1} (k=4,5,\cdots,n)$ satisfying (N4$'$)(a) and (N4$'$)(b). We need to show that in quotient category $\mathcal{Z}/\mathcal{D}$ the corresponding (N4$'$)(a) and (N4$'$)(b) are satisfied.

By Lemma \ref{3.2}, (N4$'$)(a) holds. For (N4$'$)(b), we first note that the morphism $\left(
                                                                                                                           \begin{smallmatrix}
                                                                                                                                f_3 \\
                                                                                                                                \varphi_3 \\
                                                                                                                               \end{smallmatrix}
                                                                                                                               \right)
:X_3\rightarrow X_4\oplus Y_3$ is $\mathcal{D}$-monic. In fact, for any morphism $f:X_3\rightarrow D$ with $D\in\mathcal{D}$, there exists a morphism $g:Y_2\rightarrow D$ with $ff_2=g\varphi_2$ since $\varphi_2$ is $\mathcal{D}$-monic. Now $g\varphi_2f_1=ff_2f_1=0$, which implies that there exists a morphism $h:Y_3\rightarrow D$ with $g=hg_2$. Note that $\varphi_3f_2=g_2\varphi_2$, thus $(f-h\varphi_3)f_2=0$. Then there exists a morphism $i:X_4\rightarrow D$ with $f-h\varphi_3=if_3$. Hence $f=(i\ h)\left(
          \begin{smallmatrix}
            f_3 \\
            \varphi_3 \\
          \end{smallmatrix}
        \right)
$. Therefore the $n$-angle (2.1) induces an $n$-angle in $\mathcal{Z}/\mathcal{D}$ with the last morphism $\underline{e_n}$ satisfying $d_n''\cdot e_n=\Sigma f_2\cdot h_n$, where $d_n'': TX_3\rightarrow \Sigma X_3$.

To complete the proof, it suffices to check that $\underline{e_n}=T\underline{f_2}\cdot\underline{c_n}$ and $\underline{c_n}\underline{\psi_n}=T\underline{f_1}\cdot \underline{b_n}$. Let $T\underline{f_2}=\underline{i_2}$, then by definition we have $\Sigma f_2\cdot d_n'=d_n'' i_2$. Now $d_n'' i_2 c_n=\Sigma f_2\cdot d_n' c_n=\Sigma f_2\cdot h_n=d_n''e_n$, which implies that $e_n-i_2 c_n$ factors through some object in $\mathcal{D}$, thus $\underline{e_n}=\underline{i_2} \underline{c_n}=T\underline{f_2}\cdot\underline{c_n}$. Similarly let $T\underline{f_1}=\underline{i_1}$, then $\Sigma f_1\cdot d_n=d_n' i_1$. Note that $d_n' c_n \psi_n=h_n\psi_n=\Sigma f_1\cdot g_n=\Sigma f_1\cdot d_n b_n=d_n' i_1 b_n$. Thus $c_n \psi_n-i_1 b_n$ factors through some object in $\mathcal{D}$ and $\underline{c_n}\underline{\psi_n}=T\underline{f_1}\cdot \underline{b_n}$.
\end{proof}

\begin{rem}
In Theorem \ref{3.4}, if $n=3$, then we can recover a theorem of Iyama-Yoshino \cite[Theorem 4.2]{[IY]}.
\end{rem}

Before stating a corollary, we give some definitions.
Let $\mathcal{C}$ be an $n$-angulated category and $\mathcal{Z}$ an $n$-angulated subcategory. Denote by $\mathcal{E}$ the class of all $n$-angles in $\mathcal{Z}$ of the form $$X_1\xrightarrow{f_1}X_2\xrightarrow{f_2}X_3\xrightarrow{f_3}\cdots\xrightarrow{f_{n-2}}X_{n-1}\xrightarrow{f_{n-1}}X_n\xrightarrow{f_n}\Sigma X_1$$
where we say $f_1$ is {\em admissible monomorphism} and $f_{n-1}$ is {\em admissible epimorphism}. An object $I\in \mathcal{Z}$ is called $\mathcal{E}$-{\em injective} if for every admissible monomorphism $f:X\rightarrow Y$ the sequence $\mathcal{Z}(Y,I)\xrightarrow{\mathcal{Z}(f,I)}\mathcal{Z}(X,I)\rightarrow 0$ is exact. Denote by $\mathcal{I}$ the full subcategory of $\mathcal{E}$-injectives. We say $\mathcal{Z}$ {\em has enough $\mathcal{E}$-injectives} if for each object $X\in\mathcal{Z}$ there exists an $n$-angle
$$X\xrightarrow{f_1}I_1\xrightarrow{f_2}I_2\xrightarrow{f_3}\cdots\xrightarrow{f_{n-2}}I_{n-2}\xrightarrow{f_{n-1}}Y\xrightarrow{f_n}\Sigma X$$
in $\mathcal{E}$ with all $I_i\in\mathcal{I}$. The notion of {\ em having enough $\mathcal{E}$-projectives} is defined dually.
If $\mathcal{Z}$ has enough $\mathcal{E}$-injectives, enough $\mathcal{E}$-projectives and if $\mathcal{E}$-injectives and $\mathcal{E}$-projectives coincide, then we say $\mathcal{Z}$ is {\em Frobenius}.

The following result is a higher analogue of \cite[Theorem 7.2]{[B]}. We can also compare it with \cite[Theorem 5.11]{[J]}.

\begin{cor}
Let $\mathcal{C}$ be an $n$-angulated category and $\mathcal{Z}$ an $n$-angulated subcategory. If $\mathcal{Z}$ is Frobenius, then the quotient category $\mathcal{Z}/\mathcal{I}$ is an $n$-angulated category.
\end{cor}

\begin{proof}
It is easy to see that $(\mathcal{Z},\mathcal{Z})$ is an $\mathcal{I}$-mutation pair. Note that $\mathcal{Z}$ is an $n$-angulated subcategory implies that $\mathcal{Z}$ is extension-closed, thus the result follows from Theorem \ref{3.4}.
 \end{proof}

\end{document}